\documentclass[reqno,centertags,11pt]{amsart}
\usepackage{verbatim}
\usepackage{tikz, enumerate, array, amssymb, amsmath}
\usepackage{graphicx, epsfig }
\usepackage{amsthm}
\usepackage{amsbsy}
\usepackage{amsfonts}

\usepackage[pagebackref,colorlinks]{hyperref}

\usepackage{mathtools}

\newcommand{\Hmm}[1]{\leavevmode{\marginpar{\tiny%
$\hbox to 0mm{\hspace*{-0.5mm}$\leftarrow$\hss}%
\vcenter{\vrule depth 0.1mm height 0.1mm width \the\marginparwidth}%
\hbox to 0mm{\hss$\rightarrow$\hspace*{-0.5mm}}$\\\relax\raggedright #1}}}
\newcommand{\R}{{\mathbb{R}}}

\newcommand{\C}{{\mathbb{C}}}

\newcommand{\f}{\frac}

\newcommand{\beq}{\begin{equation}}
\newcommand{\eeq}{\end{equation}}
\newcommand{\bdm}{\begin{displaymath}}
\newcommand{\edm}{\end{displaymath}}
\newcommand{\ba}{\begin{align}}
\newcommand{\ea}{\end{align}}
\newcommand{\bpf}{\begin{proof}}
\newcommand{\epf}{\end{proof}}

\newcommand{\la}{\langle}
\newcommand{\ra}{\rangle}

\newcommand{\veps}{\varepsilon}
\newcommand{\vphi}{\varphi}

\newcommand{\dav}{{d_{\mathrm{av}}}}

\allowdisplaybreaks
\newcommand{\calC}{\mathcal{C}}

\newcommand{\calS}{{\mathcal{S}}}

\newcommand{\mchi}[1]{\chi_{\raisebox{-2.3pt}{$\scriptstyle #1$}}}

\newtheorem{theorem}{Theorem}
\newtheorem{proposition}[theorem]{Proposition}
\newtheorem{lemma}[theorem]{Lemma}
\newtheorem{corollary}[theorem]{Corollary}

\theoremstyle{definition}

\newtheorem{remark}[theorem]{Remark}

\newcounter{theoremi}[theorem]

\numberwithin{theorem}{section}
\numberwithin{equation}{section}

\newcounter{assumptions}

\newcounter{smalllist}

\newcounter{listi}
\newenvironment{theoremlist}{\begin{list}{{\rm(\roman{listi})}}{%
\setlength{\topsep}{0mm}\setlength{\parsep}{0mm}\setlength{\itemsep}{0mm}%
\setlength{\labelwidth}{1.5em}\setlength{\leftmargin}{1.7em}\usecounter{listi}%
}}{\end{list}}

\newcounter{smallenum}

\begin{document}

\title[Averaging of DMNLS]{Averaging of dispersion managed nonlinear Schr\"odinger equations}
\author[M.--R. Choi, Y.--R. Lee]{Mi--Ran Choi$^\dag$, Young--Ran Lee$^\ddag$}
%% Mi-Ran's address
\address{$^\dag$ Research Institute for Basic Science, Sogang University, 35 Baekbeom--ro (Sinsu--dong),
    Mapo-gu, Seoul 04107, South Korea.}%
\email{rani9030@sogang.ac.kr}

% Younghoon Kang and Young-Ran's address

\address{$^\ddag$ Department of Mathematics, Sogang University, 35 Baekbeom--ro (Sinsu--dong),
    Mapo--gu, Seoul 04107, South Korea.}%
\email{younglee@sogang.ac.kr}

%\keywords{Non-linear Schr\"odinger equation, decay of eigenfunctions}
%\subjclass[2000]{35B20, 35B40, 35P30 }
%\date{\today, version \jobname}
%\thanks{Keywords: Nonlocal NLS, dispersion management, well--posedness, orbital stability}
%\thanks{\textit{MSC2020 classification}.35Q55, 35Q60, 35A01}
%\thanks{\copyright 2020 by the authors. Faithful reproduction of this article,
    %   in its entirety, by any means is permitted for non-commercial purposes}
%\thanks{Supported in part by the Alfried Krupp von Bohlen and Halbach Foundation  and NSF-grant DMS-0803120 (D.H.) and National Research Foundation of Korea(NRF) grant funded by the Korea government(MSIP) No. 2011-0013073 and No. 2014R1A1A2058848 (Y.-R.L.).}
%\keywords{nonlinear Schr\"odinger equation, dispersion management, well--posedness, averaging }
%\subjclass[2020]{35Q55, 35Q60, 35A01, 35B27}
\date{\today}%, version \jobname }

\begin{abstract}
We consider the dispersion managed power-law nonlinear Schr\"odinger(DM NLS) equations with a small parameter $\veps>0$ and the averaged equation, which are used in optical fiber communications. We prove that the solutions of DM NLS equations converge to the solution of the averaged equation in $H^1(\R)$ as $\veps$ goes to zero. Meanwhile, in the positive average dispersion, we obtain the global existence of the solution to DM NLS equation in $H^1(\R)$ for sufficiently small $\veps>0$, even when the exponent of the nonlinearity is beyond the mass--critical power.
\end{abstract}

\maketitle

%%%%%%%%%%%%%%%%%%%%%%%%%%%%%%%%%%
\section{Introduction }\label{introduction}
%%%%%%%%%%%%%%%%%%%%%%%%%%%%%%%%%%

We consider the nonlinear Schr\"odinger equation
\beq \label{eq:intro}
i\partial_t u +d(t) \partial_x^2 u +|u|^{\alpha} u = 0, \notag
\eeq
where $u=u(x,t)$, $x, t \in \R,$ is a complex--valued function, $d(t)$ a periodic real--valued function, and $\alpha>0$. Such an equation arises naturally as an envelope equation for electromagnetic wave propagation along the cable used in optical fiber communications, see, e.g., \cite{SS,TURITSYN2003}.
Here, $x$ denotes the (retarded) time and $d(t)$ the dispersion at position $t$ along the cable to be specified below.

The technique called dispersion management was invented in 1980, see \cite{Lin:80}, with the idea of creating rapidly varying dispersion with alternating sections of positive and negative dispersion in fibers.
Such rapid variation was successful in transfering the data at ultra-high speed over long distances, see, e.g., \cite{Ablowitz:98,Gabitov:96,Gabitov:96b,Mamyshev:99}.
For more detailed information on this technique, see \cite{TURITSYN2012135}.

The periodic modulation of the dispersion in the strong dispersion regime is given by
\beq\label{eq:d(t)}
d(t)= \dav +\f{1}{\veps}d_0 \left(\f{t}{\veps}\right),  \notag
\eeq
where $d_0$ is a periodic function of mean zero,
$\dav \in \R$ the average dispersion over one period, and $\veps$ a small positive parameter.
 In this regime, we consider the Cauchy problem
\beq \label{eq:intro_main}
\begin{cases}
   i\partial_t u +\Bigl(\dav +\f{1}{\veps}d_0 \left(\f{t}{\veps}\right)\Bigr) \partial_x^2 u +
  |u|^{\alpha} u = 0,\\
   u(x,0)=\vphi(x),
\end{cases}
\eeq
where $d_0$ is assumed to be a $2$-periodic function with $d_0=\mchi{[0,1)}- \mchi{[1,2)}$ on $[0,2)$.
Given $\varepsilon>0$, by the standard argument, it can be shown that the Cauchy problem \eqref{eq:intro_main} is globally well--posed in $H^1(\R)$ for appropriate values of $\alpha$. Indeed, for any $\dav\in\R$, if $\alpha>0$ and $\vphi\in H^1(\R)$, then a local solution exists in $H^1(\R)$. Moreover,  if $\dav \neq 0$, then by the mass and energy conservations it is easy to see that the solution is global under the additional condition $\alpha<4$ only when $\dav>0$. For the vanishing average dispersion, we use the mass conservation and the regularity of the $L^2$ solution to get the global existence provided $0< \alpha<4$. To obtain such results, one should consider the properties of the linear equation associated with \eqref{eq:intro_main}, see \cite{ASS, CKL} for more details.

\begin{comment}
Now we discuss the well--posedness in $H^1(\R)$ for \eqref{eq:intro transformed equation}.
The global well-posdeness result for \eqref{eq:intro_main}, which is equivalent of \eqref{eq:intro transformed equation}, already proved in \cite{ASS}  in the fast dispersion management $d(t)=\dav+ d_0(t/\varepsilon)$. See also \cite{CKL}, where it is proved under the cubic nonlinearity and the lumped amplification.
\end{comment}

Now we change the variables $u=T_{D(t/\veps)}v$ in \eqref{eq:intro_main} to obtain
\beq\label{eq:intro transformed equation}
\begin{cases}
   i\partial_t v + \dav \partial_x^2 v+ T_{D(t/\veps)}^{-1}\Bigl(|T_{D(t/\veps)}v|^{\alpha}T_{D(t/\veps)}v\Bigr)=0,\\
   v(x,0)=\vphi(x),
\end{cases}
\eeq
where
$D(t)=\int_0^t d_0(t')dt'$ and
$T_t$ is the solution operator for the free Schr\"{o}dinger equation in spatial dimension one.
We consider the two--scale asymptotic expansion for the solution $v$ of \eqref{eq:intro transformed equation}, that is,
\[
v(t)=\sum_{j=0}^{\infty}\veps^jv_j\left(t, \frac{t}{\veps}\right),
\]
where all $v_j=v_j(t,\tau)$ are $2$-periodic in $\tau$. Then  we see that  $v_0$ is constant in $\tau$ at order $\veps^{-1}$ and
\begin{equation}\label{eq:constant term}
i \partial_t v_0+ i \partial_\tau v_1+\dav \partial_x^2 v_0+T_{D(\tau)}^{-1}\Bigl(\left|T_{D(\tau)}v_0\right|^\alpha T_{D(\tau)}v_0\Bigr)=0
\end{equation}
at order $\veps^{0}$.
Averaging \eqref{eq:constant term} with respect to $\tau$ over one period, we have
\begin{equation}\label{eq:integral form}
i \partial_t v_0+\dav\partial_x^2 v_0+\frac{1}{2}\int_{0}^{2} T_{D(\tau)}^{-1}\Bigl(\left|T_{D(\tau)}v_0\right|^\alpha T_{D(\tau)}v_0\Bigr)d\tau=0. \notag
\end{equation}
Furthermore, if we use the change of variables $D(\tau)=r$, we have the following averaged equation of \eqref{eq:intro transformed equation}
\begin{equation} \label{eq:intro_averaged eq}
     i\partial_t v
         + d_{\text{av}} \partial_x^2  v
            + \int_0^1
                T_{r}^{-1}\bigl(|T_{r} v|^{\alpha}T_{r} v\bigr)
              \, dr = 0,
 \end{equation}
where $v_0$ is replaced by $v$. For more information on the averaging process regarding dispersion management, see, e.g., \cite{Ablowitz:98, Gabitov:96, Gabitov:96b}.

The Cauchy problem of the averaged equation \eqref{eq:intro_averaged eq} is globally well--posed  in $H^1(\R)$ when $0< \alpha< 8$ for $\dav>0$; $0< \alpha< 4$ for $\dav=0$; $\alpha>0$ for $\dav < 0$, see \cite{AK, CHLW}. It is remarkable that this averaged equation has the $H^1$ global solution even for $4\leq \alpha<8$ when $\dav>0$ in contrast to the classical focusing NLS. In \cite{AK}, the $H^s$ theory for $s \ge 0$ in the case of the Kerr nonlinearity was established, while it was shown in \cite{CHLW} that the problem is globally well--posed in $H^1(\R)$ and $L^2(\R)$ when $\dav \neq 0$ and  $\dav =0$, respectively, for more general nonlinearities including even saturated nonlinearities.
Furthermore, since a local solution in $H^1(\R)$ exists for every $\alpha>0$ regardless of the sign of $\dav$, see \cite{CHLW}, in the case $\dav=0$, one can obtain the $H^1$ global solution using $H^1$ regularity argument from the $L^2$ solution, see, e.g., \cite{Cazenave}.

Our interest is to analyze the asymptotic behavior of the solutions, $v_\veps$, for \eqref{eq:intro transformed equation} on the maximal life time of the solution $v$ for \eqref{eq:intro_averaged eq} as $\varepsilon\to 0$. When $\alpha=2$, the averaging procedure is first rigorously justified in \cite{ZGJT01}. More precisely, it is shown that for $\varepsilon>0$, the solutions of \eqref{eq:intro transformed equation}
and  \eqref{eq:intro_averaged eq} with the initial datum in $H^s(\R)$, $s$ sufficiently large, stay $\varepsilon$--close in $H^{s-3}$ for a long time in $O(\varepsilon^{-1})$.
Recently, the authors with Y. Kang, in \cite{CKL}, improved this result by verifying the averaging procedure in $H^1(\R)$ where the solutions exist for the initial datum in $H^1(\R)$, for $\alpha=2$. The operator associated with this Kerr nonlinearity is multi--linear, which is crucial to get the averaging theorems in \cite{ CKL,ZGJT01}.
In this paper, we extend the result in \cite{CKL} to the power--law nonlinearities by overcoming the difficulty caused by the fact that the operator for the general power $\alpha$ is not multi--linear. Furthermore, it follows from the global well--posedness of the averaged equation in  \cite{CHLW} and our main theorem, Theorem \ref{thm:intro averaging}, that if $\dav>0$, then, even for $4\le \alpha <8$, the Cauchy problem \eqref{eq:intro transformed equation} has a global $H^1(\R)$ solution for sufficiently small $\veps>0$.
The main theorem is

\begin{theorem}\label{thm:intro averaging}
Let $\dav\in \R$, $\alpha\geq 2$, and the initial datum $\vphi\in H^1(\R)$. For each $\veps>0$, denote by $v_\veps$ the maximal solution of \eqref{eq:intro transformed equation} and by $v$ the solution of the averaged equation \eqref{eq:intro_averaged eq} defined on the maximal interval $(-T_-, T_+)$.
Then, given $0<M<\min\{ T_-, T_+\}$, the solution  $v_\veps$ exists on $[-M, M]$ for sufficiently small $\veps>0$. Moreover,
  \beq
\lim_{\veps \to 0} \|v_\veps -v\| _{L^\infty([-M, M],H^{1}(\R))} =0. \notag
  \eeq
\end{theorem}

We have the following immediate corollary using $u=T_{D(t/\varepsilon)} v$.
\begin{corollary}
Under the assumptions of Theorem \ref{thm:intro averaging}, let $u_\varepsilon$ be the maximal solution of \eqref{eq:intro_main}. Then given  $0<M<\min\{ T_-, T_+\}$, the solution  $u_\varepsilon$ exists on $[-M, M]$ for sufficiently small $\veps>0$.
Moreover,
\beq
\lim_{\veps \to 0} \|u_\veps -T_{D(t/\varepsilon)}^{-1}v\| _{L^\infty([-M, M],H^{1}(\R))} =0. \notag
\eeq
\end{corollary}

\begin{remark}
 \begin{theoremlist}
 \item Let $0<\alpha<4$ when $\dav\geq0$; $\alpha>0$ when $\dav<0$, which are naturally assumed. Then, we know that the solution $v_\varepsilon$ of \eqref{eq:intro transformed equation} for every $\veps>0$ and the solution $v$ of \eqref{eq:intro_averaged eq} are globally defined in $H^1(\R)$.
     If, in addition, $\alpha \ge 2$, then Theorem \ref{thm:intro averaging} yields that, for each $0<M<\infty$, $v_\veps$ converge to $v$ in $L^\infty([-M, M],H^{1}(\R))$ as $\veps\to 0$.
\item In the case $\dav>0$, the range of $\alpha$ for the global existence increases. Even for $4\leq \alpha<8$, by \cite{CHLW}, it is known that
the solution $v$ of \eqref{eq:intro_averaged eq} globally exists in $H^1(\R)$, i.e., $T_+=T_-=\infty$. Thus, it follows from Theorem \ref{thm:intro averaging} that $v_\veps$ is also globally defined for sufficiently small $\veps$ and that for each $0<M<\infty$, $v_\veps \to v$ in $L^\infty([-M, M],H^{1}(\R))$ as $\veps\to 0$.
\item The restriction on $\alpha$, $\alpha\geq 2$, comes from Lemma \ref{lem:boundQ''}. Such a restriction can be removed if we use the initial datum in $H^3(\R)$.
\item Theorem \ref{thm:intro averaging} can be proved similarly even when the coefficient of the nonlinear term in \eqref{eq:intro_main} is a bounded periodic function of $t$ with the same period of $d$, see \cite{CKL} for a special case with the Kerr nonlinearity. Such nonlinearities arise in the presence of fiber loss and amplification, see, e.g., \cite{A}.
\end{theoremlist}
\end{remark}

\medskip

The paper is organized as follows. In Section \ref{sec:preliminary}, we introduce some notations and gather the bounds of the nonlinearities in two main equations, \eqref{eq:intro transformed equation} and \eqref{eq:intro_averaged eq}.
%and prove the local existence of a solution for \eqref{eq:intro transformed equation}.
In Section \ref{sec:Averaging theorem}, we prove the local existence of a solution for \eqref{eq:intro transformed equation} and the main theorem, Theorem \ref{thm:intro averaging}.

%%%%%%%%%%%%%%%%%%%%%%%%%%%%%%%%%%
\section{Preliminary Results}\label{sec:preliminary}
%%%%%%%%%%%%%%%%%%%%%%%%%%%%%%%%%%

Let us start by introducing some notations. The spaces $L^p(\R)$ for $1 \le p \le \infty$ and $H^s(\R)$ for $s\in \R$ are the usual Lebesgue and Sobolev spaces with norms $\|\cdot\|_{L^p}$ and $\|\cdot\|_{H^s}$, respectively.
%we use $L^p(\R)$ to denote the Banach space of functions $f$ whose norm
% \[
% \|f\|_{L^p} :=\left(\int _\R |f(x)|^p dx\right)^{\f1p}
 % \]
%is finite with the essential supremum instead when $p=\infty$.
 % The space $L^2(\R)$ is a Hilbert space with scalar product given by $\la f,g\ra=\int_\R f(x) \overline{g(x)} dx$.
We use $L_t^q(J, L_x^p(I))$, for $1 \le p, q < \infty$ and intervals $I, J\subset \R$, to denote the Banach space of functions $u$ with the mixed norm
$$
\|u\|_{L_t^q(J, L_x^p(I))}:=\left(\int _J \left(\int _I |u(x,t)|^p dx\right)^{\f{q}{p}}dt\right)^{\f{1}{q}}\ .
$$
If $p=\infty$ or $q=\infty$, use the essential supremum instead. For simple notations, $L^q(J,  L^p)$ is used for $L_t^q(J, L_x^p(\R))$.
%We say that $u \in L^q_{\loc}(J,  L^p)$ when $u \in L^q(\widetilde{J},  L^p)$ for every bounded interval $\widetilde{J} \subset J$.
%

For a Banach space $X$ with norm $\|\cdot\|_X$ and an interval $J$, we use $\calC(J, X)$ and $\calC^1(J, X)$ to denote the space of all continuous functions $u:J \to X$ and the space of all continuously differentiable functions, respectively. For a compact $J$, $\calC(J, X)$ is the Banach space with norm
$$
\|u\|_{\calC(J, X)}=\sup _{t\in J}\|u(t)\|_{X}.
$$

The solution operator  $T_t$ for the free Schr\"{o}dinger equation in spatial dimension one is unitary on $H^s(\R)$ for $s\in \R$ and, therefore,
$$
\|T_t f\|_{H^s}=\|f\|_{H^s}
%\quad \mbox{and} \quad \|T_t f\|_{H^1}=\|f\|_{H^1}.
$$
for all $f\in H^s(\R)$ and $t\in \R$. The notation $f \lesssim g$ is used when there is a positive constant $C$ such that $f \le Cg$.
%We also use the standard ‘big oh’ notation $O$.

\medskip

Next, we gather some estimates for the nonlinear terms of two main equations, \eqref{eq:intro transformed equation} and \eqref{eq:intro_averaged eq},
$$
 Q(s,f):=T_{D(s)}^{-1} \Bigl(|T_{D(s)}f|^{\alpha} T_{D(s)}f\Bigr)
$$
and
$$
\la Q\ra (f) := \int_0^1  T_r^{-1} \Bigl(|T_r f|^{\alpha}  T_r f\Bigr)dr
$$
defined for $s\in\R$ and $f\in H^1(\R)$.

\begin{lemma}\label{lem:boundQ}
Let $\alpha> 0$. Then
\beq \label{bound:Q}
\sup_{s\in \R}\|Q(s,f)\|_{H^1}\lesssim \|f\|_{H^1}^{\alpha+1}
\eeq
and
\beq \label{bound:averageQ}
\|\la Q \ra (f)\|_{H^1}\lesssim \|f\|_{H^1}^{\alpha+1}
\eeq
for all $f\in H^1(\R)$.
\end{lemma}
\begin{proof}
Note that
\[
\||h|^\alpha h\|_{H^1}\lesssim  \|h\|_{H^1}^{\alpha+1}
\]
since $\||h|^\alpha h\|_{L^2}\leq \|h\|^{\alpha}_{L^\infty}\|h\|_{L^2}\le\|h\|_{H^1}^{\alpha+1}$ and $\|(|h|^\alpha h)'\|_{L^2}\lesssim  \|h\|_{L^\infty}^\alpha \| h'\|_{L^2}\le\|h\|_{H^1}^{\alpha+1}$ for all $h\in H^1(\R)$, by the embedding $L^\infty(\R) \hookrightarrow H^1(\R)$.  Using the fact that $T_{D(s)}$ and $\partial_x$ commute and that $T_{D(s)}$ is unitary in $H^1(\R)$, we have \eqref{bound:Q}.
Similarly, we prove  bound \eqref{bound:averageQ} using Minkowski’s inequality.
\end{proof}

\begin{lemma}\label{lem:boundQ''}
Let $\alpha\geq 2$. Then
\beq
\sup_{s\in \R}\|\partial_x^2 Q(s,f)\|_{H^1}\lesssim \|f\|_{H^3}^{\alpha+1} \notag
\eeq
and
\beq \label{bound:averageQ2}
\|\partial_x ^2 \la Q \ra (f)\|_{H^1}\lesssim  \|f\|_{H^3}^{\alpha+1} \notag
\eeq
for all $f\in H^3(\R)$.
%where the implicit constant depends only on $\alpha$.
\end{lemma}
\begin{proof}
Note that if $h\in H^3(\R)$, then $\|(|h|^\alpha h)''\|_{H^1}\lesssim \|h\|_{H^3}^{\alpha+1}$
since
$$
|(|h|^\alpha h)''|\lesssim (|h|^{\alpha-1}|h'|^2+|h|^\alpha |h''|),
$$
$$
|(|h|^\alpha h)'''|\lesssim \left(|h|^{\alpha-2}|h'|^3+|h|^{\alpha-1} |h'||h''| + |h|^\alpha |h'''|\right),
$$
and $\alpha \ge 2$.
%, the Schwartz space of infinitely smooth, rapidly decreasing functions.
Thus, an argument similar to that used in proving Lemma \ref{lem:boundQ} completes the proof since  $T_{D(s)}$ is unitary in $H^3(\R)$, also.
\end{proof}

%%%%%%%%%%%%%%%%%%%%%%%%%%%%%%%%%%
\section{Averaging Theorem}\label{sec:Averaging theorem}
%%%%%%%%%%%%%%%%%%%%%%%%%%%%%%%%%%

In this section, we prove the main theorem, Theorem \ref{thm:intro averaging}. First, we establish the following two lemmas.

\begin{lemma}\label{lem:Lipschiz}
Let $\alpha \geq 1$ and $M>0$. Then
\beq \label{conv:Q}
\begin{aligned}
&\sup_{\veps>0}\int_0^M \left\|Q\left(\f{t}{\varepsilon}, v_1(t)\right) -Q\left(\f{t}{\varepsilon}, v_2(t)\right)\right\|_{H^1}dt \\
 \lesssim &\left(\|v_1\|^\alpha_{L^\infty([0,M],H^1)}+\|v_2\|^\alpha_{L^\infty([0,M],H^1)}\right)\|v_1-v_2\|_{L^\infty([0,M],H^1)}
\end{aligned}
\eeq
and
\beq \label{conv:averageQ}
\begin{aligned}
& \int_0^M\left\|\la Q\ra (v_1(t)) - \la Q\ra (v_2(t))\right\|_{H^1}dt \\
 \lesssim & \left(\|v_1\|^\alpha_{L^\infty([0,M],H^1)}+\|v_2\|^\alpha_{L^\infty([0,M],H^1)}\right)\|v_1-v_2\|_{L^\infty([0,M],H^1)}
\end{aligned}
\eeq
for all $v_1, v_2\in \calC([0,M], H^1(\R))$.%, where the implicit constants depend only on $\alpha$ and $M$.
\end{lemma}

\begin{proof}
We prove \eqref{conv:Q} only since \eqref{conv:averageQ} can be proved analogously.
First, we prove
\beq\label{conv:L^2first}
\begin{aligned}
&\sup_{\veps>0} \int_0^M \left\|Q\left(\f{t}{\varepsilon}, v_1(t)\right) -Q\left(\f{t}{\varepsilon}, v_2(t)\right)\right\|_{L^2}dt \\
\lesssim &\left(\|v_1\|^\alpha_{L^\infty([0,M],H^1)}+\|v_2\|^\alpha_{L^\infty([0,M],H^1)}\right)\|v_1-v_2\|_{L^\infty([0,M],L^2)}.
\notag
\end{aligned}
\eeq
Note that
\beq\label{eq:difference of cubes}
||z_1|^{\alpha}z_1-|z_2|^{\alpha}z_2|\lesssim
(|z_1|^{\alpha}+|z_2|^{\alpha})|z_1-z_2| \quad\mbox{for all }z_1, z_2\in \C. \notag
\eeq
It follows from
%\eqref{eq:difference of cubes} and
%Thus, by
the embedding $L^\infty(\R) \hookrightarrow H^1(\R)$ that
\begin{align*}
&\sup_{\veps>0} \int_0^M \left\|Q\left(\f{t}{\varepsilon}, v_1(t)\right) -Q\left(\f{t}{\varepsilon}, v_2(t)\right)\right\|_{L^2} dt \\
=& \sup_{\veps>0} \int_0^M \left\|  |T_{D(\f t \varepsilon)}v_1(t)|^\alpha T_{D(\f t \varepsilon)}v_1(t)  -|T_{D(\f t \varepsilon)}v_2(t)|^\alpha T_{D(\f t \varepsilon)}v_2(t)\right\|_{L^2} dt\\
 \lesssim &\int_0^M \left (\| v_1(t)\|_{H^1}^\alpha + \| v_2(t)\|_{H^1}^\alpha \right) \|v_1(t)-v_2(t)\|_{L^2}dt \label{ineq:Q_L^2}\\
\lesssim &\left(\|v_1\|^\alpha_{L^\infty([0,M],H^1)}+\|v_2\|^\alpha_{L^\infty([0,M],H^1)}\right)
\|v_1-v_2\|_{L^\infty([0, M], L^2)}.
\end{align*}
Next, to complete the proof of \eqref{conv:Q}, we first observe that
\[
%\label{ineq:derivative of nonlinear term}
\left|\left(|f|^{\alpha}f- |g|^{\alpha}g\right)'\right| \lesssim |f|^{\alpha} | f'-g'|+  \bigl||f|^{\alpha}-|g|^{\alpha}\bigr| | g'|+\bigl||f|^{ \alpha-1}f^2-|g|^{\alpha-1}g^2\bigr| | g'|
\]
for continuously differentiable functions $f, g$ on $\R$.
Thus,
\begin{align}
&\int_0 ^M \left\|\partial_x \left(Q\left(\f t \varepsilon, v_1(t)\right)- Q\left(\f t \varepsilon, v_2(t)\right)\right)\right\|_{L^2} dt \notag\\
=& \int_0 ^M  \left\|\partial_x \left(|T_{D(\f t \varepsilon)}v_1(t)|^\alpha T_{D(\f t \varepsilon)}v_1(t) - |T_{D(\f t \varepsilon)}v_2(t)|^\alpha T_{D(\f t \varepsilon)}v_2(t)\right) \right\|_{L^2} dt \notag\\
\lesssim &\int_0 ^M  \left \||T_{D(\f t \varepsilon)}v_1(t)|^\alpha\partial_x \left( T_{D(\f t \varepsilon)}v_1(t) -T_{D(\f t \varepsilon)}v_2(t)\right) \right\|_{L^2}dt \notag\\
&   + \int_0 ^M \left\| \left(|T_{D(\f t \varepsilon)}v_1(t)|^\alpha -|T_{D(\f t \varepsilon)}v_2(t)|^\alpha\right )\partial_x T_{D(\f t \varepsilon)}v_2(t) \right\|_{L^2} dt  \label{ineq:second}\\
& + \int_0 ^M \left\|\left(|T_{D(\f t \varepsilon)}v_1(t)|^{\alpha-1}(T_{D(\f t \varepsilon)}v_1(t))^2 -|T_{D(\f t \varepsilon)}v_2(t)|^{\alpha-1} (T_{D(\f t \varepsilon)}v_2(t))^2 \right)|\partial_x T_{D(\f t \varepsilon)}v_2(t)|\right\|_{L^2} dt. \notag
\end{align}
It is easy to see that the first term of \eqref{ineq:second} is bounded as
\[
\begin{aligned}
& \sup_{\veps>0} \int_0 ^M  \left \||T_{D(\f t \varepsilon)}v_1(t)|^\alpha\partial_x \left( T_{D(\f t \varepsilon)}v_1(t) -T_{D(\f t \varepsilon)}v_2(t)\right)  \right\|_{L^2}dt  \\
&\lesssim \|v_1\|_{L^\infty([0, M], H^1)}^\alpha \|\partial_x(v_1-v_2)\|_{L^\infty([0, M], L^2)}.
\end{aligned}
\]
%for all $n\geq N$.
Since the third term can be bounded similarly to the second term, we bound the second term only. First, we note that, for $z_1,z_2 \in \C$,
\beq \label{ineq:zw}
\left||z_1|^\alpha-|z_2|^\alpha\right|\lesssim |z_1-z_2| (|z_1|^{\alpha-1}+|z_2|^{\alpha-1})\notag
\eeq
since $\alpha \ge 1$.
Thus, similarly as before, we have
%and the fact that $\|f\|_{L^\infty}\leq \|f'\|_{L^2}^{1/2} \|f\|_{L^2}^{1/2}$ to see
\begin{align*}
&\sup_{\veps>0}\int_0 ^M \left\|\big(|T_{D(\f t \varepsilon)}v_1(t)|^\alpha -|T_{D(\f t \varepsilon)}v_2(t)|^\alpha \big) \partial_x T_{D(\f t \varepsilon)}v_2(t)\right \|_{L^2} dt\\
%&\leq \left\||T_{D(\f t \varepsilon)}v_n(t)|^\alpha -|T_{D(\f t \varepsilon)}v(t)|^\alpha \right\|_{L^\infty}
%\left\|\partial_x T_{D(\f t \varepsilon)}v(t)\right\|_{L^2}\\
\lesssim & \int_0 ^M \|v_1(t) - v_2(t)\|_{H^1}\left(\|v_1(t)\|_{H^1}^{\alpha-1}+\|v_2(t)\|_{H^1}^{\alpha-1}\right)  \left\|\partial_x  v_2(t)\right\|_{L^2} dt\\
\lesssim & \left(\|v_1\|^\alpha_{L^\infty([0,M],H^1)}+\|v_2\|^\alpha_{L^\infty([0,M],H^1)}\right) \|v_1- v_2\|_{L^\infty([0,M], H^1)},
\end{align*}
which completes the proof.
%Similarly, the last term of \eqref{ineq:second} converges to $0$ as $n\to \infty$, uniformly in $\varepsilon$, by the bound for $z_1, z_2 \in \C$
%\[
%||z_1|^{\alpha-2}z_1^2-|z_2|^{\alpha-2}z_2^2|\lesssim |z_1-z_2| (|z_1|^{\alpha-1}+|z_2|^{\alpha-1}).
%\]
\end{proof}

Using Lemma \ref{lem:Lipschiz}, we have the following lemma which is the key ingredient in our work. This is inspired by \cite{FMS,MS} where the nonlinear Schr\"odinger equation with strong confinement was analyzed.%, see also \cite{MS}.

\begin{lemma}\label{lemma:averaging}
Let $\alpha\geq 2 $ and $M>0$. If $v\in \calC([-M,M], H^1(\R))$, then
\beq\label{conv:averaging}
\sup_{t\in[-M, M]}\left\|\int_0^t e^{i \dav (t-s)\partial_x^2 }\left[Q \left(\f{s}{\veps}, v(s)\right)- \la Q\ra (v(s))\right] ds\right\|_{H^1} \to  0
\eeq
as $\veps \to 0$.
\end{lemma}

\begin{proof}
We consider positive times only.
It follows from Lemma \ref{lem:Lipschiz} that if $v_1, v_2$ belong to $\calC([-M, M],H^1(\R))$, then
\begin{align*}
&\sup_{t\in[0, M]}\left\|\int_0^t e^{i \dav (t-s)\partial_x^2 }
\left(Q \left(\f{s}{\veps}, v_1(s)\right) - Q \left(\f{s}{\veps}, v_2(s)\right)\right)ds\right\|_{H^1} \\
\leq& \int_0 ^M \left\|Q \left(\f{s}{\veps}, v_1(s)\right) - Q \left(\f{s}{\veps}, v_2(s)\right)\right\|_{H^1}ds \\
\lesssim &\left(\|v_1\|^\alpha_{L^\infty([0,M],H^1)}+\|v_2\|^\alpha_{L^\infty([0,M],H^1)}\right)\|v_1-v_2\|_{L^\infty([0,M],H^1)}
\end{align*}
and
\begin{align*}
&\sup_{t\in[0, M]}\left\|\int_0^t e^{i \dav (t-s)\partial_x^2 }
\Big(\la Q\ra( v_1(s)) - \la Q\ra ( v_2(s)) \Big)ds\right\|_{H^1}\\
\lesssim &\left(\|v_1\|^\alpha_{L^\infty([0,M],H^1)}+\|v_2\|^\alpha_{L^\infty([0,M],H^1)}\right)\|v_1-v_2\|_{L^\infty([0,M],H^1)}.
\end{align*}
Therefore, by a density argument, it is enough to prove \eqref{conv:averaging} for $v\in \calC^1([0, M], \mathcal{S}(\R))$ only,
where the Schwartz space $\calS(\R)$ consists of infinitely differentiable, rapidly decreasing functions.

We define
\beq\label{def:Q}
\mathbf{Q}(\theta, f):= \int_0 ^ \theta \left[ Q(s, f )-\la Q\ra (f) \right]ds\notag
\eeq
 on $[0, \infty)\times H^1(\R)$. Note that for each $f\in H^1(\R)$, $\mathbf{Q}(\cdot, f)$ is $2$-periodic since $Q (\cdot, f)$ is a $2$-periodic function whose average is $\la Q\ra(f)$. Thus, we have
 \beq \label{ineq:supQ}
 \begin{aligned}
 \sup_{\theta \in\R} \| \mathbf{Q}(\theta, f)\|_{H^1}
 &= \sup_{\theta\in [0,2]} \|\mathbf{Q}(\theta, f)\|_{H^1}\\
 & \leq \sup_{\theta\in [0,2]} \int_0 ^\theta \left(\|Q(s, f)\|_{H^1} + \|\la Q \ra (f)\|_{H^1}\right)ds \\
 & \lesssim \|f\|_{H^1}^{\alpha+1}
 \end{aligned}
 \eeq
 for all $f\in H^1(\R)$,
where we use Lemma \ref{lem:boundQ} in the last inequality. Similarly, it follows from Lemma \ref{lem:boundQ''} that
\beq\label{ineq:supQQ}
\sup_{\theta\in \R} \left\| \partial_x^2 \mathbf{Q} \left(\theta, f \right) \right\|_{H^1}
 \lesssim \|f\|_{H^3}^{\alpha+1}
 \eeq
for all $f\in H^3(\R)$.

Now let $v\in \calC^1([0, M], \mathcal{S}(\R)) $. By a simple calculation, we have
%\beq\label{eq:derivative}
\begin{align}\label{eq:derivative}
&\f{d}{ds}\left(e^{i\dav (t-s)\partial_x^2} \mathbf{Q} \left(\f{s}{\veps}, v(s) \right)\right)\notag \\
=& i \dav e^{i\dav (t-s)\partial_x^2} \partial_x^2 \mathbf{Q}\left(\f{s}{\veps}, v(s) \right)+
\f{1}{\veps}e^{i\dav (t-s)\partial_x^2}\left(Q \left(\f{s}{\veps}, v(s) \right)-\la Q\ra (v(s))\right)
\\
 &+ e^{i\dav (t-s)\partial_x^2} \int_0 ^{\f s \veps}\f{d}{ds}\left[ T_{D(s')}^{-1}\Bigl(|T_{D(s')}v(s)|^\alpha T_{D(s')}v(s)\Bigr)
-\int_0 ^1 T_r ^{-1}\Bigl(|T_{r}v(s)|^\alpha T_rv(s)\Bigr)dr \right]ds'. \notag
\end{align}
%\eeq
Note that
\begin{align*}
&\f{d}{ds}T_{D(s')}^{-1}\Bigl(|T_{D(s')}v(s)|^\alpha T_{D(s')}v(s)\Bigr)\\
= & T_{D(s')}^{-1} \left[\f{d}{ds}\left(|T_{D(s')}v(s)|^\alpha T_{D(s')} v(s)\right)\right]\\
= &  T_{D(s')}^{-1} \left[\f{\alpha+2}{2}|T_{D(s')}v(s)|^{\alpha}T_{D(s')}\partial_t v(s)+ \f{\alpha}{2} |T_{D(s')}v(s)|^{\alpha-2}(T_{D(s')}v(s))^2 \overline{T_{D(s')}\partial_t v(s) }\right]
\end{align*}
which is $2$-periodic in $s'$. Moreover, its average over one period is
\[
\f{d}{ds}\int_0 ^1 T_r ^{-1}\Bigl(|T_{r}v(s)|^\alpha T_rv(s)\Bigr)dr.
\]
Thus, by the same argument as above, we have
\beq\label{ineq:supQQQ}
\begin{aligned}
&\sup_{\theta \in\R}\left\| \int_0 ^{\theta}\f{d}{ds}\left[ T_{D(s')}^{-1}\Bigl(|T_{D(s')}v(s)|^\alpha T_{D(s')}v(s)\Bigr)
-\int_0 ^1 T_r ^{-1}\Bigl(|T_{r}v(s)|^\alpha T_rv(s)\Bigr)dr \right]ds' \right\|_{ H^1}\\
=&\sup_{\theta \in [0,2]}\left\| \int_0 ^{\theta}\f{d}{ds}\left[ T_{D(s')}^{-1}\Bigl(|T_{D(s')}v(s)|^\alpha T_{D(s')}v(s)\Bigr)
-\int_0 ^1 T_r ^{-1}\Bigl(|T_{r}v(s)|^\alpha T_rv(s)\Bigr)dr \right]ds' \right\|_{ H^1}\\
\lesssim  & \|v(s)\|_{H^1}^{\alpha}\|\partial_tv(s)\|_{H^1}.
\end{aligned}
\eeq
Therefore, using the bounds \eqref{ineq:supQ}, \eqref{ineq:supQQ}, \eqref{ineq:supQQQ} together with \eqref{eq:derivative}, we obtain
\begin{align*}
&\left\|\int_0^\cdot e^{i\dav (\cdot-s)\partial_x^2}\left(Q \left(\f{s}{\veps}, v(s) \right)-\la Q\ra (v(s))\right) ds \right\|_{L^\infty([0,M], H^1)} \\
\leq
& \veps \sup_{t\in [0, M]} \left\| \mathbf{Q} \left(\f{t}{\veps}, v(t) \right) \right\|_{H^1}
 + \veps  |\dav|\int_0^M \left\| \partial_x^2 \mathbf{Q}\left(\f{s}{\veps}, v(s) \right)\right\|_{H^1} ds \\
& + \veps \int_0 ^M \left\| \int_0 ^{s/\veps}\f{d}{ds}\left[ T_{D(s')}^{-1}\Bigl(|T_{D(s')}v(s)|^\alpha T_{D(s')}v(s)\Bigr)
-\int_0 ^1 T_r ^{-1}\Bigl(|T_{r}v(s)|^\alpha T_rv(s)\Bigr)dr \right]ds' \right\|_{ H^1}ds \\
\lesssim  & \veps \Bigl[ \|v\|_{L^\infty([0,M],H^3)}^{\alpha+1}+\|v\|_{L^\infty([0,M],H^1)}^{\alpha}\|\partial_tv\|_{L^\infty([0,M],H^1)}\Bigr]
\end{align*}
which completes the proof.
\end{proof}

In preparation for the proof of the main theorem, Theorem \ref{thm:intro averaging}, we  provide the local existence of a solution for \eqref{eq:intro transformed equation}.
%, which is an important point to prove Theorem \ref{thm:intro averaging}.
%The crucial point here is that the right side is independent of \varepsilon.
Given $\vphi \in H^1(\R)$, we consider the Duhamel formula for \eqref{eq:intro transformed equation}
\beq \label{eq:duhamel fomula}
v_\veps(t)=e^{it \dav \partial_x^2} \vphi + i \int _0 ^t e^{i (t-s)\dav \partial_x^2} Q\left( \f{s}{\varepsilon}, v_\varepsilon (s) \right)ds.
\eeq

\begin{proposition}\label{prop:H^slocal theory}
Let $\dav\in \R$ and $\alpha>0$. For every $ K>0$, there exist $M_\pm>0$ such that if $\varepsilon>0$ and the initial datum $\vphi \in H^1(\R)$ satisfies $\|\vphi\|_{H^1}\leq K$, then there is a unique solution $v_\varepsilon\in \calC([-M_-, M_+], H^1(\R))$ of \eqref{eq:duhamel fomula}.
Moreover,
\beq
\|v_\varepsilon(t)\|_{ H^1}\leq 2 K  \quad \text{for all }\veps>0 \mbox{ and } t\in [-M_-, M_+]. \notag
\eeq
\end{proposition}
\begin{proof}
The result is quite standard
and its proof is very analogous to that of Proposition 3.5 in \cite{CHLW}.
Instead of Lemma 2.6 in \cite{CHLW}, we use Lemma \ref{lem:boundQ} and
\[
\begin{aligned}\label{est:local}
&\sup_{s\in\R}\| Q\left( s, f\right) - Q\left( s,g \right)\|_{L^2}\\
= & \sup_{s\in \R}
\||T_{D(s)}f|^{\alpha}T_{D(s)}f-|T_{D(s)}g|^{\alpha}T_{D(s)}g\|_{L^2}\\
\lesssim & \left(\|f\|_{H^1}^{\alpha} +\|g\|_{H^1}^{\alpha} \right)  \|f-g\|_{L^2},
\end{aligned}
\]
for all $f, g\in H^1(\R)$.
\end{proof}

\medskip
\begin{proof}[Proof of Theorem \ref{thm:intro averaging}]
Let $M>0$ be fixed and consider positive times only.
Let
\beq\label{assume:K}
K=2\sup_{t\in [0, M]}\|v(t)\|_{H^1}.
\eeq
Thus, since $\|\vphi\|_{H^1}=\|v(0)\|_{H^1}\leq K/2$,
 it follows from Proposition \ref{prop:H^slocal theory} that there exists  a positive $M_1$ corresponding to $K/2$, independent of $\veps$, such that
$v_\veps\in \calC([0, M_1], H^1(\R))$ and
\beq\label{sup:veps}
\sup_{\veps>0}\sup_{t \in [0, M_1]} \|v_\veps(t) \|_{H^{1}} \leq K.
\eeq

We assume that $M_1< M$.
Using Duhamel's formula, for $t\in [0, M_1]$, we write
\[
v_\veps(t)-v(t) = i \mathcal{I}_1(t)+i\mathcal{I}_2(t),
\]
where
\[
\mathcal{I}_1(t)=\int_0^t  e^{i\dav (t-s)\partial_x^2}\left[Q \left(\f{s}{\veps}, v_\veps(s)\right)
-Q \left(\f{s}{\veps}, v(s)\right)\right] ds
\]
and
\[
\mathcal{I}_2(t)=\int_0^t  e^{i\dav (t-s)\partial_x^2}\left[Q \left(\f{s}{\veps}, v(s)\right)-
 \la Q\ra (v(s))\right] ds.
\]
Then, it follows from Lemma \ref{lemma:averaging} that  %there exists a positive constant $C_2$ independent of $\veps$ such that
\beq\label{eq:I2}
\left\|\mathcal{I}_2 \right\|_{L^\infty([0, M_1], H^1)} := \eta_\varepsilon \to 0
\eeq
as $\veps\to 0$.
On the other hand, in order to estimate $\mathcal{I}_1$, use Minkowski's inequality and the same argument in the proof of Lemma \ref{lem:Lipschiz}, then we obtain
\beq\label{eq:I1}
\begin{aligned}
\left\|\mathcal{I}_1(t) \right\|_{H^1}  & \leq  \int_0^t \left\|Q \left(\f{s}{\veps}, v_\veps(s)\right)
-Q \left(\f{s}{\veps}, v(s)\right) \right\|_{H^1} ds \\
& \lesssim  \int_0^t (\|v_\veps(s) \|^{\alpha}_{H^{1}}+\|v(s)\|^{\alpha}_{H^{1}} )\|v_\veps(s) -v(s) \|_{H^1} ds\\
& \lesssim  K^\alpha \int_0^t \|v_\veps(s) -v(s) \|_{H^1} ds
\end{aligned}
\eeq
for all $0\leq t \leq M_1$, where we use \eqref{assume:K} and \eqref{sup:veps} in the last bound.
It follows from \eqref{eq:I2} and \eqref{eq:I1} that there exists a positive constant $C$, independent of $\varepsilon$, such that
\begin{align*}
  \|v_\veps(t)-v(t)\|_{H^1}
 % \le \left\|\mathcal{I}_1(t) \right\|_{H^{1}} +\left\|\mathcal{I}_2(t) \right\|_{H^{1}}
   \le \eta_\varepsilon+ C\int_0^t \|v_\veps(s) -v(s) \|_{H^1} ds
\end{align*}
for all $0\leq t \leq M_1$.
Thus, by Gronwall's inequality, we obtain
\[
\sup_{t \in [0, M_1]}\|v_\veps(t)-v(t)\|_{H^1} \leq \sup_{t\in [0, M_1]}\eta_\varepsilon e^{Ct}\leq \eta_\varepsilon e^{CM_1}\to 0
\]
as $\veps \to 0$.
Note that assuming $M_1<M$ is acceptable since, if $M_1 \ge M$, replacing $M_1$ on the above by $M$ completes the proof.

Next, since
\beq
\sup_{\varepsilon>0} \|v_\veps(M_1)\|_{H^{1}}\leq K \notag
\eeq
by \eqref{sup:veps}, it follows from Proposition \ref{prop:H^slocal theory} with the initial datum $v_\veps(M_1)$ that
the solution $v_\varepsilon$ exists on $[M_1, M_1 +M_2 ]$ for some positive $M_2$ corresponding to $K$, independent of $\varepsilon$, and
\[
\sup_{ \veps >0 } \sup_{t\in [M_1,  M_1 +M_2]}\|v_\veps(t)\|_{H^1}\leq 2K.
\]
The last inequality together with \eqref{sup:veps} yields
\beq\label{bound:veps_1}
\sup_{ \veps >0 } \sup_{t\in [0,  M_1 +M_2]}\|v_\veps(t) \|_{H^{1}} \leq 2K \notag
\eeq
and, therefore, by the same argument as above, we obtain
\beq\label{conv:M_1}
\|v_\veps - v\|_{L^\infty([0, \min\{M, M_1 +M_2\}],H^{1})}\to 0
\eeq
as $\varepsilon\to 0$.
Thus, if $M_1+M_2 \ge M$, then the proof is complete.

Now assume that $M_1+M_2 < M$. Then,
by \eqref{assume:K} and \eqref{conv:M_1}, there exists $\varepsilon_1>0$ such that
$$
\sup_{ 0< \veps \leq \varepsilon_1 } \sup_{t\in [0,  M_1 +M_2]}\|v_\veps(t)\|_{H^1}\leq K.
$$
Moreover, since $v_\veps \in \calC([0, M_1 +M_2], H^1(\R))$,
\[
\sup_{ 0< \veps \leq \veps_1 }\|v_\veps(M_1+M_2)\|_{H^1} \le K .
\]
Replacing $v_\veps(M_1)$ and $\veps>0$ by $v_\veps(M_1+M_2)$ and $0<\veps<\veps_1$, respectively, in the previous argument, we obtain
\[
\|v_\veps-v\|_{L^\infty([0, \min\{M,M_1 +2M_2\}],H^{1})}\to 0
\]
as $\veps \to 0$.
Iterating this procedure finitely many times, until the sum of $M_1$ and the integer multiple of $M_2$ gets greater than or equal to $M$,
 completes the proof.
\end{proof}

%%%%%%%%%%%%%%%%%%%%%%%%%%%%%%%%%%%%%%%%%%%%%
%%%%%%%%%%%%%%%%%%%%%%%%%%%%%%%%%%%%%%%%%%%%%

\vspace{5mm}

\appendix
\setcounter{section}{0}
\renewcommand{\thesection}{\Alph{section}}
\renewcommand{\theequation}{\thesection.\arabic{equation}}
\renewcommand{\thetheorem}{\thesection.\arabic{theorem}}

\noindent

\textbf{Acknowledgements: }

Young--Ran Lee and Mi--Ran Choi are supported by the National Research Foundation of Korea(NRF) grants funded by the Korean government (MSIT) NRF-2020R1A2C1A01010735 and (MOE) NRF-2021R1I1A1A01045900.

%This work was supported by the National Research Foundation of Korea(NRF) grant
%funded by the Korea government(MSIT) (No. 한국연구재단에서 부여한 과제 관리번호).
%%%%%%%%%%%%%%%%%%%%%%%%%%%%%%%%%%%%%%%%
 %End of appendix
%\renewcommand{\thesection}{\arabic{chapter}.\arabic{section}}
%\renewcommand{\theequation}{\arabic{chapter}.\arabic{section}.\arabic{equation}}
%\renewcommand{\thetheorem}{\arabic{chapter}.\arabic{section}.\arabic{theorem}}
%%%%%%%%%%%%%%%%%%%%%%%%%%%%%%%%%%%%%%%%%
%
%
%
%\def\cprime{$'$}

\bibliographystyle{abbrv}
\bibliography{Averaging_power_law_bibfile}

\begin{comment}

%%%%%%%%%%%%%%%%%%%%%%%%%%%%

\end{comment}

%%%%%%%%%%%%%%%%%%%%%%%%%%%%%%%%%%%%%%%%%%%%%%%%%%%5
\end{document}